\documentclass[11pt]{article}

\usepackage{amsfonts}
\usepackage{amsmath}

\newcommand{\beq}{\begin{equation}}
\newcommand{\eeq}{\end{equation}}
\newcommand{\bea}{\begin{eqnarray}}
\newcommand{\eea}{\end{eqnarray}}
\newcommand{\beas}{\begin{eqnarray*}}
\newcommand{\eeas}{\end{eqnarray*}}
\newcommand{\para}{\mathbin{\!/\mkern-5mu/\!}}

\newtheorem{theorem}{Theorem}[section]

\newtheorem{definition}[theorem]{Definition}
\newtheorem{proposition}[theorem]{Proposition}

\newtheorem{lemma}[theorem]{Lemma}
\newtheorem{remark}[theorem]{Remark}
\newtheorem{example}[theorem]{Example}
\newtheorem{examples}[theorem]{Examples}
\newtheorem{foo}[theorem]{Remarks}

\newenvironment{proof}{\addvspace{\medskipamount}\par\noindent{\it
Proof}.}
{\unskip\nobreak\hfill$\Box$\par\addvspace{\medskipamount}}

\newcommand{\ee}{\ell}

\newcommand{\bM}{\mathbb M}

\newcommand{\Ho}{\mathcal H}
\newcommand{\V}{\mathcal V}
\newcommand{\di}{\mathfrak h}
\newcommand{\M}{\mathbb M}

\title{The Lichnerowicz-Obata theorem  on sub-Riemannian manifolds with transverse symmetries}
\author{Fabrice Baudoin, Bumsik Kim}

\date{Department of Mathematics, Purdue University \\
 West Lafayette, IN, USA}

\begin{document}
\maketitle

\begin{abstract}
 We prove a lower bound for the first eigenvalue of the sub-Laplacian on sub-Riemannian manifolds with transverse symmetries. When the manifold is of $H$-type, we obtain a corresponding rigidity result: If the optimal lower bound for the first eigenvalue is reached, then  the manifold is equivalent to a 1 or a 3-Sasakian sphere.
 \end{abstract}

\tableofcontents

\section{Introduction}

The study of optimal lower bounds for  sub-Laplacians on manifolds has attracted a lot of interest in the past  few years. In particular, the most studied example has been the example of the sub-Laplacian on CR manifolds. In that case, the story goes back at least to the work by Greenleaf \cite{Gr} which has seen, since then, several improvements and variations. We mention in particular the works by Aribi-Dragomir-El Soufi \cite{ADE}, Barletta \cite{barletta}, Baudoin-Wang \cite{BW2}, Ivanov-Petkov-Vassilev \cite{IPV1,IPV5}, Li \cite{Li} and Li-Luk \cite{Li2}. Some optimal lower bounds for the first eigenvalue of sub-Laplacians also have been obtained in the context of quaternionic contact manifolds by Ivanov-Petkov-Vassilev \cite{IPV,IPV3,IPV4}. More general situations were even considered by Hladky \cite{Hl}.

\

In the present work, we obtain optimal first eigenvalue lower bounds in a large class of sub-Riemannian manifolds that encompasses as a very special case Sasakian manifolds and 3-Sasakian manifolds. This class is the class of sub-Riemannian manifolds with transverse symmetries that was introduced in  \cite{BG}. Roughly speaking, a sub-Riemannian manifold with transverse symmetries is a sub-Riemannian manifold for which the horizontal distribution admits a canonical intrinsic complement which is generated by sub-Riemannian Killing fields. The lower bound we obtain in that case improves a previous lower bound that was obtained by Baudoin-Kim in \cite{BK}. The method of \cite{BK} was to apply to an eigenfunction of the sub-Laplacian the curvature-dimension inequality  proved in \cite{BG}, and then to integrate this curvature-dimension inequality over the manifold. When used on a Riemannian manifold, this technique provides the optimal Lichnerowicz estimate. However, interestingly, this technique does not give the optimal estimate in the sub-Riemannian case and more work is needed. Our approach here, is to take advantage of the Bochner-Weitzenb\"ock formula that was recently proved in  \cite{baudoin} and to integrate this equality over the manifold. This gives an equality which when applied to an eigenfunction gives a better estimate than \cite{BK} for the first eigenvalue. In the 1 or the 3-Sasakian case, the lower bound we obtain coincides with the known optimal lower bound.

\

In the second part of the paper, we check the optimality of our lower bound, by proving a rigidity result in the spirit of Obata \cite{obata}. More precisely we prove the following result:

\begin{theorem}
Let $\M$ be a compact sub-Riemannian manifold of $H$-type with dimension $d+\di$, $d$ being the dimension of the horizontal bundle and $\di$ the dimension of the vertical bundle. Assume  that for every smooth horizontal one-form $\eta$,
\[
\langle \mathfrak{Ric}_{\mathcal{H}}(\eta),\eta \rangle_\mathcal{H^*} \ge \rho \| \eta \|^2_\mathcal{H^*},
\]
with $\rho>0$, then  the first eigenvalue $\lambda_1$ of the sub-Laplacian $-L$  satisfies
\begin{align*}
 \lambda_1 \geq  \frac{\rho d}{d-1+3 \di }.
\end{align*}
Moreover, if $\lambda_1 =  \frac{\rho d}{d-1+3 \di }$, then $\M$ is equivalent to a 1-Sasakian sphere $ \mathbb{S}^{2m+1} ( r )$ or a 3-Sasakian sphere $ \mathbb{S}^{4m+3}( r )$ for some $r>0$ and $m \ge 1$.
\end{theorem}

This result for $H$-type manifolds generalizes the corresponding theorem for Sasakian manifolds by Chang-Chiu \cite{Chiu}  and for 3-Sasakian manifolds by Ivanov-Petkov-Vassilev \cite{IPV3}. Like in the cited references, the main idea is to prove that an extremal eigenfunction $f$ for the sub-Laplacian needs to satisfy $\tilde{\nabla}^2 f =-\alpha f$, for the  Levi-Civita connection of a well chosen Riemannian extension of the sub-Riemannian metric.  We can observe that in the works  \cite{IPV1,Li} or \cite{IPV4} the Sasakian condition is not needed, it is therefore an interesting question to try to generalize our result to more general sub-Riemannian structures where the transverse symmetries condition is not assumed.

\

The paper is organized as follows. Section 2 presents the basic materials on sub-Riemannian manifolds with transverse symmetries. In particular, we present the Bochner-Weitzenb\"ock formula that was proved in \cite{baudoin}. Section 3 is devoted to the proof of the lower bound for the first eigenvalue and Section 4 proves its optimality in the context of $H$-type manifolds.

\section{The Bochner-Weitzenb\"ock formula on sub-Riemannian manifolds with transverse symmetries}

The notion of sub-Riemannian manifold with transverse symmetries was introduced in \cite{BG}. We recall here the main geometric quantities and operators related to this structure and   we refer to \cite{baudoin} and  \cite{BG} for further details. We in particular focus on the Bochner-Weitzenb\"ock formula that was proved in \cite{baudoin}.

\

Let $\M$ be a smooth, connected  manifold with dimension $d+\di$. We assume that $\bM$ is equipped with a bracket generating distribution $\mathcal{H}$ of dimension $d$ and a fiberwise inner product $g_\mathcal{H}$ on that distribution.
The distribution $\mathcal{H}$ is referred to as the set of \emph{horizontal directions}, while a vector field which is tangent to $\Ho$ is said to be horizontal.

\begin{definition}
It is said that $\M$ is a sub-Riemannian manifold with transverse symmetries if there exists an $\di$- dimensional Lie algebra $\mathcal{V}$ of sub-Riemannian Killing vector fields such that for every $x \in \bM$,
 \[
 T_x \bM= \mathcal{H}(x) \oplus \mathcal{V}(x).
 \]
\end{definition}

We recall that a vector field $Z$ is said to be a sub-Riemannian Killing vector field if the flow it generates locally preserves the horizontal distribution and induces a $g_\mathcal{H}$-isometry. Also $\mathcal{V}$ denotes the distribution referred to as the set of \emph{vertical directions}.
 The choice of an inner product $g_{\mathcal{V}}$ on the Lie algebra $\mathcal{V}$ naturally endows $\bM$ with a one-parameter family of Riemannian metrics that makes the decomposition $\mathcal{H} \oplus \mathcal{V}$ orthogonal:
\[
g_{\varepsilon}=g_\mathcal{H} \oplus \frac{1}{\varepsilon}  g_{\mathcal{V}}, \quad \varepsilon >0.
\]
 For notational convenience, we will often use the notation $\langle \cdot, \cdot \rangle_\varepsilon$, resp. $\langle \cdot ,\cdot \rangle_\mathcal{H}$,  resp $\langle \cdot ,\cdot \rangle_\mathcal{V}$, instead of $g_\varepsilon$, resp. $g_\mathcal{H}$, resp. $g_\mathcal{V}$.  We can extend $g_\mathcal{H}$ on $T_x\M \times T_x \M$ by the requirement that $g_\mathcal{H}(u,v)=0$ whenever $u$ or $v$ is in $\mathcal{V}(x)$. We similarly extend $g_\mathcal{V}$. Hence for any $u \in T_x\M$,
 \[
 \| u \|_{\varepsilon}^2=\| u\|_\mathcal{H}^2 +\frac{1}{\varepsilon} \| u\|_\mathcal{V}^2.
 \]

 The volume measure obtained as a product of the horizontal volume measure determined by $g_\mathcal{H} $ and the volume measure determined by $g_\mathcal{V} $ will be denoted by $\mu$ and is our reference measure on $\M$.

\

The following  connection was introduced in \cite{BG}.

\begin{proposition}[See \cite{BG}]
There exists a unique connection $\nabla$ on $\mathbb{M}$ satisfying the following properties:
\begin{itemize}
\item[(i)] $\nabla g_\varepsilon =0$, $\varepsilon >0$;
\item[(ii)] If $X$ and $Y$ are horizontal vector fields, $\nabla_{X} Y$ is horizontal;
\item[(iii)] If $Z \in \mathcal{V}$, $\nabla Z=0$;
\item[(iv)] If $X,Y$ are horizontal vector fields and $Z \in \mathcal{V}$,
the torsion vector field $T(X,Y)$ is vertical and $T(X, Z)=0$.
\end{itemize}
\end{proposition}

Intuitively $\nabla$ is the connection which coincides with  the Levi-Civita connection of the Riemannian metric $g_\varepsilon$ on the horizontal bundle $\mathcal{H}$ and that parallelizes the Lie algebra $\mathcal{V}$.

At every point $x\in \bM$, we can find a local  frame of vector fields $\{X_1,\cdots,X_d, Z_1, \cdots, Z_\di\}$ such that on a neighborhood of $x$:
 \begin{itemize}
\item[(a)] $\{X_1,\cdots,X_d \}$ is a $g_\mathcal{H}$-orthonormal basis of $\mathcal{H}$;
\item[(b)] $\{Z_1, \cdots, Z_\di \}$ is a $g_\mathcal{V}$-orthonormal basis of the Lie algebra $\mathcal{V}$;
\end{itemize}
Such a frame will be called a local adapted frame.

The sub-Laplacian on $\M$ is the second-order differential operator which is given in a local adapted frame by
\begin{align}\label{sublap}
L=\sum_{i=1}^d \nabla_{X_i}\nabla_{X_i} -\nabla_{\nabla_{X_i} X_i}.
\end{align}

By declaring a one-form horizontal (resp. vertical) if it vanishes on the vertical bundle $\mathcal{V}$ (resp. on the horizontal bundle $\mathcal{H}$), the splitting of the tangent space
 \[
 T_x \bM= \mathcal{H}(x) \oplus \mathcal{V}(x)
 \]
 gives a splitting of the cotangent space
  \[
 T^*_x \bM= \mathcal{H}^*(x) \oplus \mathcal{V}^*(x).
 \]

If $\{X_1,\cdots,X_d, Z_1, \cdots, Z_\di\}$ is a local  adapted frame, the dual frame will be denoted $\{\theta_1,\cdots,\theta_d, \nu_1, \cdots, \nu_\di\}$ and referred to as a local adapted coframe. With a slight abuse of notations, for $\varepsilon>0$, the metric on $ T^*_x \bM$ that makes $\{\theta_1,\cdots,\theta_d, \frac{1}{\sqrt{\varepsilon} } \nu_1, \cdots, \frac{1}{\sqrt{\varepsilon}} \nu_\di\}$ orthonormal will still be denoted $g_\varepsilon$ or $\langle \cdot, \cdot \rangle_\varepsilon$. This metric on the cotangent bundle can thus be written
 \begin{equation}\label{rdgh}
g_{\varepsilon}=g_\mathcal{H^*} \oplus \varepsilon  g_{\mathcal{V^*}}, \quad \varepsilon >0,
 \end{equation}
where $g_\mathcal{H^*}$ (resp. $g_\mathcal{V^*}$)  is the metric on $\mathcal{H}^*$ (resp. $\mathcal{V}^*$) that makes $\{\theta_1,\cdots,\theta_d\}$ (resp. $\{ \nu_1, \cdots, \nu_\di\}$ ) orthonormal. We use similar notations and conventions as before so that for every $\eta$ in $T^*_x \M$,
\[
\| \eta \|^2_{\varepsilon} =\| \eta \|_\mathcal{H^*}^2+\varepsilon \| \eta \|_\mathcal{V^*}^2.
\]

\

 We now introduce some tensors that will play an important role in the sequel.
 We define $\mathfrak{Ric}_{\mathcal{H}}: T^*_x \bM \to T^*_x \bM$ as the  symmetric linear map on one forms such that for every smooth functions $f,g$,
\[
\langle  \mathfrak{Ric}_{\mathcal{H}} (df), dg \rangle_\mathcal{H^*}=\mathbf{Ricci} (\nabla_\mathcal{H} f ,\nabla_\mathcal{H} g),
\]
where $\mathbf{Ricci}$ is the Ricci curvature of the connection $\nabla$ and $\nabla_\mathcal{H}$ the horizontal gradient (projection of the gradient on the horizontal distribution $\mathcal{H}$).  Similarly, we will denote by $\nabla_\mathcal{V}$ the vertical gradient, that is the projection of the gradient on the vertical bundle. In a local adapted frame $\{X_1,\cdots,X_d, Z_1, \cdots, Z_\di\}$, we have thus
\[
\nabla_\mathcal{H} f= \sum_{i=1}^d (X_i f) X_i,
\]
\[
\nabla_\mathcal{V} f= \sum_{m=1}^\di (Z_m f) Z_m.
\]
  and
  \[
  \mathbf{Ricci} (\nabla_\mathcal{H} f ,\nabla_\mathcal{H} g)=\sum_{n=1}^d g_\mathcal{H} ( \mathbf{R}(\nabla_\mathcal{H} f ,X_i)X_i, \nabla_\mathcal{H} g),
  \]
  where $\mathbf{R}$ is the Riemannian curvature tensor: $\mathbf{R}(X_i,X_j)X_k= \nabla_{X_i} \nabla_{X_j}X_k-\nabla_{X_j} \nabla_{X_i}  X_k -\nabla_{[X_i,X_j]} X_k$.

  \

  For $Z \in \mathcal{V}$, we consider  the unique skew-symmetric map $J_Z$ defined on the horizontal bundle $\mathcal{H}$ such that for every horizontal vector fields $X$ and $Y$,
\begin{align}\label{Jmap}
\langle J_Z (X),Y\rangle_\mathcal{H}= \langle Z,T(X,Y) \rangle_{\mathcal{V}}.
\end{align}
We can then extend $J_Z$ to the whole tangent space $T_x\M$ by imposing that $J_Z (V)=0$ whenever $V$ is a vertical vector field.
If $(Z_m)_{1 \le m \le \di}$ is a $g_\V$-orthonormal basis of the Lie algebra $\mathcal{V}$, the operator $\sum_{m=1}^\di J_{Z_m}^*J_{Z_m}=-\sum_{m=1}^\di J^2_{Z_m}:T_x\M \to T_x \M$ does not depend on the choice of the basis and will concisely be denoted by $-\mathbf{J}^2$. We can note that in the case where $\M$ is a Sasakian manifold, $\mathbf{J}^2=-\di \mathbf{Id}_\mathcal{H}$. Though originally defined on vector fields we will also consider $-\mathbf{J}^2$ as the linear map $T^*_x\M\to T_x^*\M $ defined by
\[
    \langle -\mathbf{J}^2 (\theta_i),\theta_j \rangle_{\Ho^*} = \langle -\mathbf{J}^2(X_i) , X_j  \rangle_\Ho,\quad\quad 1\le i,j\le d
\]
Then $-\mathbf{J}^2$ is defined to be $0$ on vertical one-forms.

\

If $V$ is a horizontal vector field, then we consider an operator $\mathfrak{T}^\varepsilon_V$ on smooth sections of the cotangent bundle given by
\[
\mathfrak{T}^\varepsilon_V \eta =-\sum_{j=1}^d \eta (T(V,X_j)) \theta_j +\frac{1}{2 \varepsilon} \sum_{m=1}^\di \eta( J_{Z_m} V) \nu_m
\]
in a local frame.
It is easily seen that $\mathfrak{T}^\varepsilon_V$ is a skew-symmetric operator for the  metric $g_{2\varepsilon}$ that was previously defined on one-forms by \eqref{rdgh}.

\

If $\eta$ is a one-form, we define the horizontal gradient in a local adapted frame of $\eta$ as the $(0,2)$ tensor
\[
\nabla_\mathcal{H} \eta =\sum_{i=1}^d \nabla_{X_i} \eta \otimes \theta_i.
\]
Similarly, we will use the notation
\[
\mathfrak{T}^\varepsilon_\mathcal{H} \eta =\sum_{i=1}^d \mathfrak{T}^\varepsilon_V \eta  \otimes \theta_i.
\]

We finally recall the following definition that was introduced in \cite{BG}:

\begin{definition}
The sub-Riemannian manifold $\M$ is said to be of Yang-Mills type, if the horizontal divergence of the torsion vanishes that is for every horizontal vector field $X$, and every adapted local frame
\[
 \sum_{\ee=1}^d(\nabla_{X_\ee} T) (X_\ee,X)=0.
\]
\end{definition}

There are many interesting examples of Yang-Mills sub-Riemannian manifolds with transverse symmetries (see \cite{BG}). Sasakian and 3-Sasakian manifolds are examples of Yang-Mills sub-Riemannian manifolds. Though not identical, the Yang-Mills condition can be compared to the divergence free  torsion condition that was considered in \cite{IPV1}.

\

The following Bochner Weitzenb\"ock formula was proved in \cite{baudoin} to which we refer for further details.

\begin{theorem}[Bochner-Weitzenb\"ock formula \cite{baudoin}]\label{bochner}
Assume that $\M$ is a sub-Riemannian manifold with transverse symmetries of Yang-Mills type. For $\varepsilon >0$, we consider the  $g_{2 \varepsilon}$-self-adjoint operator which is defined on one-forms by
\[
\square_\varepsilon=-(\nabla_\mathcal{H} -\mathfrak{T}_\mathcal{H}^\varepsilon)^* (\nabla_\mathcal{H} -\mathfrak{T}_\mathcal{H}^\varepsilon)-\frac{1}{2 \varepsilon} \mathbf{J}^2 - \mathfrak{Ric}_{\mathcal{H}}.
\]
Then, for every smooth function $f$ on $\M$,
\[
d(Lf)=\square_\varepsilon (df),
\]
and for any smooth one-form $\eta$,
\[
\frac{1}{2} L \| \eta \|_{2\varepsilon}^2 -\langle \square_\varepsilon \eta , \eta \rangle_{2\varepsilon}=  \| \nabla_{\mathcal{H}} \eta  -\mathfrak{T}_\mathcal{H}^\varepsilon  \eta \|_{2\varepsilon}^2 +\left\langle \mathfrak{Ric}_{\mathcal{H}}(\eta)+\frac{1}{2 \varepsilon} \mathbf{J}^2 (\eta), \eta \right\rangle_\mathcal{H^*}.
\]
\end{theorem}

In the previous statement $(\nabla_\mathcal{H} -\mathfrak{T}_\mathcal{H}^\varepsilon)^*$ is understood as an adjoint for the $g_{2 \varepsilon}$-metric and  it is easily seen (see \cite{baudoin})  that in a local adapted frame, we have
\[
-(\nabla_\mathcal{H} -\mathfrak{T}_\mathcal{H}^\varepsilon)^* (\nabla_\mathcal{H} -\mathfrak{T}_\mathcal{H}^\varepsilon)=\sum_{i=1}^d (\nabla_{X_i} -\mathfrak{T}^\varepsilon_{X_i})^2 - ( \nabla_{\nabla_{X_i} X_i}-  \mathfrak{T}^\varepsilon_{\nabla_{X_i} X_i}) ,
\]
and for any smooth one-form $\eta$,
\[
\| \nabla_{\mathcal{H}} \eta  -\mathfrak{T}_\mathcal{H}^\varepsilon  \eta \|^2_{2\varepsilon}=\sum_{i=1}^d  \| \nabla_{X_i} \eta  -\mathfrak{T}^\varepsilon_{X_i} \eta \|_{2\varepsilon}^2 .
\]

\section{Lichnerowicz estimate}

From now on, we consider a compact Yang-Mills sub-Riemannian manifold $\M$ with transverse symmetries and adopt the conventions and notations of the previous section. In particular $L$ denotes the sub-Laplacian on $\M$.  In this section, we prove the following result.

\begin{theorem}\label{lichne}
Assume that for every smooth horizontal one-form $\eta$,
\[
\langle \mathfrak{Ric}_{\mathcal{H}}(\eta),\eta \rangle_\mathcal{H^*} \ge \rho_1 \| \eta \|^2_\mathcal{H^*},\quad  \left\langle -\mathbf{J}^2(\eta), \eta \right\rangle_\mathcal{H^*} \le \kappa \| \eta \|^2_\mathcal{H^*},
\]
and that for every $Z \in \mathcal{V}$,
\[
\mathbf{Tr} ( J_Z^* J_Z) \ge \rho_2 \| Z \|_\mathcal{V}^2,
\]
with $\rho_1,\rho_2 >0$ and $\kappa \ge 0$. Then the first eigenvalue $\lambda_1$ of the sub-Laplacian $-L$  satisfies
\begin{align*}
 \lambda_1 \geq  \frac{\rho_1}{1-\frac{1}{d}+\frac{3\kappa}{\rho_2}}.
\end{align*}
\end{theorem}

Before we prove the result, we briefly discuss the argument that was used in \cite{BK} to quickly  get, under the same assumptions, a lower bound on $\lambda_1$ which is less sharp.

 If $f$ is a smooth function on $\M$, then we have from Theorem \ref{bochner}
\[
\frac{1}{2} L \| df \|_{2\varepsilon}^2 -\langle  d(Lf) , df \rangle_{2\varepsilon}=  \| \nabla_{\mathcal{H}} df  -\mathfrak{T}_\mathcal{H}^\varepsilon  df \|_{2\varepsilon}^2 +\left\langle \mathfrak{Ric}_{\mathcal{H}}(df)+\frac{1}{2 \varepsilon} \mathbf{J}^2(df), df \right\rangle_\mathcal{H^*}.
\]
Integrating this equality over $\M$ and using the assumptions
\[
\langle \mathfrak{Ric}_{\mathcal{H}}(\eta),\eta \rangle_\mathcal{H^*} \ge \rho_1 \| \eta \|^2_\mathcal{H^*},\quad  \left\langle -\mathbf{J}^2(\eta), \eta \right\rangle_\mathcal{H^*} \le \kappa \| \eta \|^2_\mathcal{H^*},
\]
we deduce
\begin{align*}
-\int_\M \langle  d(Lf) , df \rangle_{2\varepsilon} \ge \int_\M \| \nabla_{\mathcal{H}} df  -\mathfrak{T}_\mathcal{H}^\varepsilon  df \|_{2\varepsilon}^2 +\left( \rho_1-\frac{\kappa}{2\varepsilon}\right)\int_\M \| df \|_\mathcal{H^*}^2 .
\end{align*}
An integration by parts of left hand side of the inequality gives then

\begin{align}\label{integratedCD}
\int_\M (Lf)^2 -2\varepsilon \int_\M \langle  d(Lf) , df \rangle_\mathcal{V^*}   \ge \int_\M \| \nabla_{\mathcal{H}} df  -\mathfrak{T}_\mathcal{H}^\varepsilon  df \|_{2\varepsilon}^2 +\left( \rho_1-\frac{\kappa}{2\varepsilon}\right)\int_\M \| df \|_\mathcal{H^*}^2 .
\end{align}

Now, a straightforward application of the Cauchy-Schwarz inequality yields  the pointwise lower bound
\begin{align}\label{CS}
\| \nabla_{\mathcal{H}} df  -\mathfrak{T}_\mathcal{H}^\varepsilon  df \|_{\mathcal{H^*}}^2 \ge \frac{1}{d} (Lf)^2 +\frac{1}{4} \rho_2 \| df \|_\mathcal{V^*}^2.
\end{align}
Coming back to \eqref{integratedCD}, we infer then
\[
\frac{d-1}{d}  \int_\M (Lf)^2 -2\varepsilon \int_\M \langle  d(Lf) , df \rangle_\mathcal{V^*}  \ge \left( \rho_1-\frac{\kappa}{2\varepsilon}\right)\int_\M \| df \|_\mathcal{H^*}^2+\frac{1}{4} \rho_2 \int_\M \| df \|_\mathcal{V^*}^2.
\]
In particular, if $Lf=-\lambda_1 f$, then we obtain
\[
\frac{d-1}{d} \lambda_1^2  \int_\M f^2 +2\varepsilon \lambda_1 \int_\M \| df \|_\mathcal{V^*}^2  \ge \left( \rho_1-\frac{\kappa}{2\varepsilon}\right)\lambda_1 \int_\M f^2+\frac{1}{4} \rho_2 \int_\M \| df \|_\mathcal{V^*}^2.
\]
Choosing $\varepsilon$ such that $ 2\varepsilon \lambda_1=\frac{1}{4} \rho_2$ yields
\begin{align*}
 \lambda_1 \geq  \frac{\rho_1}{1-\frac{1}{d}+\frac{4\kappa}{\rho_2}}.
\end{align*}
This is not the optimal lower bound we are looking for.  It is possible to improve this lower bound from \eqref{integratedCD} by first integrating by parts the term $ \int_\M \| \nabla_{\mathcal{H}} df  -\mathfrak{T}_\mathcal{H}^\varepsilon  df \|_{2\varepsilon}^2 $ and, then using Cauchy-Schwarz inequality. The key lemma is the following:

\begin{lemma}\label{LMM}
For $f \in C^\infty(\M)$,
\begin{align*}
\int_\M \| \nabla_{\mathcal{H}} df  -\mathfrak{T}_\mathcal{H}^\varepsilon  df \|_{2\varepsilon}^2 =& \int_\M \| \nabla_{\mathcal{H}} df  -\mathfrak{T}_\mathcal{H}^\varepsilon  df \|_{\mathcal{H^*}}^2 +2\varepsilon \int_\M \left\| \nabla_{\mathcal{H}} df  -\frac{3}{2}\mathfrak{T}_\mathcal{H}^\varepsilon  df \right\|_{\mathcal{V^*}}^2 \\
 & +\frac{1}{2} \int_\M \mathbf{Tr} ( J^*_{\nabla_\mathcal{V}f} J_{\nabla_\mathcal{V}f})-\frac{5}{2} \varepsilon \int_\M \|\mathfrak{T}_\mathcal{H}^\varepsilon  df \|_\mathcal{V^*}^2.
\end{align*}
\end{lemma}

\begin{proof}
Using the definition $\mathfrak{T}_\mathcal{H}^\varepsilon $ together with the Yang-Mills assumption, we see that
\begin{align}\label{IPP}
\int_\M \langle \nabla_\mathcal{H} df , \mathfrak{T}_\mathcal{H}^\varepsilon (df) \rangle_{\mathcal{V}^*}=\frac{1}{4 \varepsilon} \int_\M \mathbf{Tr} ( J^*_{\nabla_\mathcal{V}f} J_{\nabla_\mathcal{V}f}).
\end{align}
As a consequence, we obtain
\begin{align}\label{IPP2}
 & \int_\M \| \nabla_{\mathcal{H}} df  -\mathfrak{T}_\mathcal{H}^\varepsilon  df \|_{2\varepsilon}^2 \notag  \\
 =&\int_\M \| \nabla_{\mathcal{H}} df  -\mathfrak{T}_\mathcal{H}^\varepsilon  df \|_{\mathcal{H}^*}^2+2 \varepsilon \int_\M \| \nabla_{\mathcal{H}} df  -\mathfrak{T}_\mathcal{H}^\varepsilon  df \|_{\mathcal{V}^*}^2 \notag  \\
 =&\int_\M \| \nabla_{\mathcal{H}} df  -\mathfrak{T}_\mathcal{H}^\varepsilon  df \|_{\mathcal{H}^*}^2+2 \varepsilon \int_\M \| \nabla_{\mathcal{H}} df   \|_{\mathcal{V}^*}^2 -4 \varepsilon \int_\M \langle \nabla_\mathcal{H} df , \mathfrak{T}_\mathcal{H}^\varepsilon (df) \rangle_{\mathcal{V}^*} +2 \varepsilon \int_\M \| \mathfrak{T}_\mathcal{H}^\varepsilon df   \|_{\mathcal{V}^*}^2
\end{align}
By using \eqref{IPP}, the trick is now to write
\begin{align*}
\int_\M \langle \nabla_\mathcal{H} df , \mathfrak{T}_\mathcal{H}^\varepsilon (df) \rangle_{\mathcal{V}^*}&=\frac{3}{2} \int_\M \langle \nabla_\mathcal{H} df , \mathfrak{T}_\mathcal{H}^\varepsilon (df) \rangle_{\mathcal{V}^*} -\frac{1}{2} \int_\M \langle \nabla_\mathcal{H} df , \mathfrak{T}_\mathcal{H}^\varepsilon (df) \rangle_{\mathcal{V}^*} \\
 &=\frac{3}{2} \int_\M \langle \nabla_\mathcal{H} df , \mathfrak{T}_\mathcal{H}^\varepsilon (df) \rangle_{\mathcal{V}^*} - \frac{1}{8 \varepsilon} \int_\M \mathbf{Tr} ( J^*_{\nabla_\mathcal{V}f} J_{\nabla_\mathcal{V}f}).
\end{align*}
Coming back to \eqref{IPP2} and completing the squares gives
\begin{align*}
\int_\M \| \nabla_{\mathcal{H}} df  -\mathfrak{T}_\mathcal{H}^\varepsilon  df \|_{2\varepsilon}^2 =& \int_\M \| \nabla_{\mathcal{H}} df  -\mathfrak{T}_\mathcal{H}^\varepsilon  df \|_{\mathcal{H}^*}^2 +2\varepsilon \int_\M \left\| \nabla_{\mathcal{H}} df  -\frac{3}{2}\mathfrak{T}_\mathcal{H}^\varepsilon  df \right\|_{\mathcal{V}^*}^2 \\
 & +\frac{1}{2} \int_\M \mathbf{Tr} ( J^*_{\nabla_\mathcal{V}f} J_{\nabla_\mathcal{V}f})-\frac{5}{2}\varepsilon \int_\M \|\mathfrak{T}_\mathcal{H}^\varepsilon  df \|_{\mathcal{V}^*}^2.
 \end{align*}
\end{proof}

We are now in position to complete the proof of Theorem \ref{lichne}.

\begin{proof}
Using the previous Lemma, Cauchy-Schwarz inequality and the assumptions
\[
\langle \mathfrak{Ric}_{\mathcal{H}}(\eta),\eta \rangle_{\mathcal{H}^*} \ge \rho_1 \| \eta \|^2_{\mathcal{H}^*},\quad  \left\langle -\mathbf{J}^2(\eta), \eta \right\rangle_{\mathcal{H}^*} \le \kappa \| \eta \|^2_{\mathcal{H}^*}, \quad \mathbf{Tr} ( J_Z^* J_Z) \ge \rho_2 \| Z \|_\mathcal{V}^2
\]
we get the lower bound
\[
\int_\M \| \nabla_{\mathcal{H}} df  -\mathfrak{T}_\mathcal{H}^\varepsilon  df \|_{2\varepsilon}^2 \ge \frac{1}{d}\int_\M  (Lf)^2 +\frac{3}{4} \rho_2 \int_\M  \| df \|_{\mathcal{V}^*}^2-\frac{5}{8 \varepsilon} \kappa \int_\M  \| df \|_{\mathcal{H}^*}^2.
\]
From \eqref{integratedCD}, we know that
\begin{align}\label{BB}
\int_\M (Lf)^2 -2\varepsilon \int_\M \langle  d(Lf) , df \rangle_{\mathcal{V}^*}   \ge \int_\M \| \nabla_{\mathcal{H}} df  -\mathfrak{T}_\mathcal{H}^\varepsilon  df \|_{2\varepsilon}^2 +\left( \rho_1-\frac{\kappa}{2\varepsilon}\right)\int_\M \| df \|_{\mathcal{H}^*}^2 .
\end{align}
We thus deduce
\[
\frac{d-1}{d} \int_\M (Lf)^2 -2\varepsilon \int_\M \langle  d(Lf) , df \rangle_{\mathcal{V}^*}  \ge \left( \rho_1-\frac{9\kappa}{8\varepsilon}\right)\int_\M \| df \|_{\mathcal{H}^*}^2+\frac{3}{4} \rho_2 \int_\M  \| df \|_{\mathcal{V}^*}^2.
\]
Now if $f$ satisfies $Lf =-\lambda_1 f$, we get
\[
\frac{d-1}{d} \lambda_1^2 \int_\M f^2 +2\varepsilon \lambda_1 \int_\M  \| df \|_{\mathcal{V}^*}^2  \ge \left( \rho_1-\frac{9\kappa}{8\varepsilon}\right)\lambda_1 \int_\M f^2+\frac{3}{4} \rho_2 \int_\M  \| df \|_{\mathcal{V}^*}^2.
\]
Choosing $\varepsilon$ such that
\[
2\varepsilon \lambda_1 =\frac{3}{4} \rho_2,
\]
the desired lower bound on $\lambda_1$ is obtained.
\end{proof}

\section{The Obata sphere theorem on $H$-type manifolds}

In this section we prove the optimality of the lower bound for the first eigenvalue of the sub-Laplacian on a special class of Yang-Mills manifolds by obtaining a rigidity result in the spirit of the Obata sphere theorem.

We first introduce the following definition inspired from the notion of $H$-type groups that was introduced by  Kaplan \cite{Kaplan}.

\begin{definition}
Let $\M$ be a sub-Riemannian manifold with transverse symmetries of Yang-Mills type. We will say that $\M$ is of $H$-type if for every $Z \in \mathcal{V}$, $\| Z \|_\mathcal{V}=1$, the map $J_Z$ is orthogonal, that is, $ \langle J_Z (X),J_Z(Y) \rangle_\Ho = \langle X,Y \rangle_\Ho $ for $X,Y\in \Ho(x)$.
\end{definition}

Sasakian or 3-Sasakian manifolds are examples of $H$-type manifolds. If $\M$ is a $H$-type sub-Riemannian manifold,  it is immediate from the definition that for $Z,Z' \in \mathcal{V}$,
\[
J_Z J_{Z'}+J_{Z'}J_Z =-2 \langle Z, Z' \rangle_\mathcal{V} \mathbf{Id}_{\mathcal{H}}.
\]
In particular, we have
\[
J^2_Z=-\|Z\|^2_{\mathcal{V} } \mathbf{Id}_{\mathcal{H}}.
\]

In this section, we prove the following result:

\begin{theorem}\label{obata}
Let $\M$ be a compact sub-Riemannian manifold of $H$-type. Assume  that for every smooth horizontal one-form $\eta$,
\[
\langle \mathfrak{Ric}_{\mathcal{H}}(\eta),\eta \rangle_{\mathcal{H}^*} \ge \rho \| \eta \|^2_{\mathcal{H}^*},
\]
with $\rho>0$, then  the first eigenvalue $\lambda_1$ of the sub-Laplacian $-L$  satisfies
\begin{align*}
 \lambda_1 \geq  \frac{\rho d}{d-1+3 \di }.
\end{align*}
Moreover, if $\lambda_1 =  \frac{\rho d}{d-1+3 \di }$, then $\M$ is equivalent to a 1-Sasakian sphere $ \mathbb{S}^{2m+1} ( r )$ or a 3-Sasakian sphere $ \mathbb{S}^{4m+3}( r )$ for some $r>0$ and $m \ge 1$.
\end{theorem}

To put things in perspective, we pause a little and describe the sub-Riemannian geometry of the 1 and 3 Sasakian spheres (see for instance \cite{BW1,BW3} for more details) and precise what we mean by equivalent in the previous theorem.
\begin{itemize}
\item The sub-Riemannian geometry of the standard 1-Sasakian sphere $\mathbb{S}^{2m+1}(1)$  is induced from the Riemannian structure of the complex projective space $\mathbb{CP}^m$ by  the Hopf fibration $ \mathbf{U}(1) \to \mathbb{S}^{2m+1} \to \mathbb{CP}^m$. The sub-Laplacian $L$ is then the lift of the Laplace-Beltrami operator on $\mathbb{CP}^m$. In that case, $\lambda_1=2m$.
\item The sub-Riemannian geometry of the standard 3-Sasakian sphere $\mathbb{S}^{4m+3}$  is induced from the Riemannian structure of the quaternionic projective space $\mathbb{HP}^m$ by  the quaternionic Hopf fibration $ \mathbf{SU}(2) \to \mathbb{S}^{4m+3} \to \mathbb{HP}^m$. The sub-Laplacian $L$ is then the lift of the Laplace-Beltrami operator on  $\mathbb{HP}^n$. In that case, $\lambda_1=m$.

\end{itemize}

In the previous theorem, we use the following notion of equivalence for sub-Riemannian manifolds with transverse symmetries: Two sub-Riemannian manifolds with transverse symmetries $(\M_1, \mathcal{H}_1, \mathcal{V}_1)$ and $(\M_2, \mathcal{H}_2, \mathcal{V}_2)$  are said to be equivalent if there exists a diffeomorphism $\M_1 \to \M_2$ that induces an isometry between the horizontal distributions $\mathcal{H}_1$ and $\mathcal{H}_2$ and a Lie algebra isomorphism between $\mathcal{V}_1$ and $\mathcal{V}_2$.



 \

We now discuss the cases that were already known  in the literature. As we pointed out Sasakian manifolds are of $H$-type. In that case $\di=1$ and the  lower bound becomes
\[
 \lambda_1 \geq  \frac{\rho d}{d+2 }.
\]
This estimate was  obtained by Greenleaf \cite{Gr} (see also \cite{barletta}). The estimate is optimal and the corresponding Obata's type rigidity result was obtained in \cite{Chiu} (see also \cite{IPV1} and \cite{Li}).

The other case that was studied in the literature is the case of 3-Sasakian manifolds for which $\di=3$. The lower bound is then
\[
 \lambda_1 \geq  \frac{\rho d}{d+8 }.
\]
This bound was proved in \cite{IPV,IPV3} and the corresponding rigidity result was obtained in \cite{IPV4}.

\

 We now turn to the proof of Theorem \ref{obata}. From now on, in the sequel, $\M$ will be a compact sub-Riemannian manifold of $H$-type such that for every smooth horizontal one-form $\eta$,
\[
\langle \mathfrak{Ric}_{\mathcal{H}}(\eta),\eta \rangle_{\mathcal{H}^*} \ge \rho \| \eta \|^2_{\mathcal{H}^*},
\]
with $\rho>0$. Since $\M$ is of $H$-type, we have
\[
\left\langle -\mathbf{J}^2(\eta), \eta \right\rangle_{\mathcal{H}^*} = \di \| \eta \|^2_{\mathcal{H}^*},
\]
and  for every $Z \in \mathcal{V}$,
\[
\mathbf{Tr} ( J_Z^* J_Z) = d \| Z \|_\mathcal{V}^2.
\]
From Theorem  \ref{lichne}, we get therefore the lower bound

\begin{align*}
 \lambda_1 \geq  \frac{\rho d}{d-1+3 \di }.
\end{align*}

The key lemma in our rigidity result is the following result:

\begin{lemma}
Let $f \in C^\infty(\M)$ such that $Lf=-\lambda_1 f$ with $ \lambda_1 =  \frac{\rho d}{d-1+3 \di }$. Then $f$ satisfies
\begin{align} \label{equ:Hessian extremal}
 \nabla^2 f(X,Y) = -\frac{\lambda_1}{d} f  \langle X,Y \rangle_\mathcal{H} -\frac{1}{2} T(X,Y) f ,\quad \forall X,Y\in\mathcal H .
\end{align}
and
\begin{align} \label{equ:extremal equations}
\nabla^2 f (X,Z)=\frac{2 \lambda_1}{\rho_2} J_Z(X)f, \quad \forall X \in\mathcal H, Z \in \mathcal{V} .
\end{align}
\end{lemma}

 \begin{proof}
 From \eqref{BB} we have
 \[
 \int_\M (Lf)^2 -2\varepsilon \int_\M \langle  d(Lf) , df \rangle_{\mathcal{V}^*}   \ge \int_\M \| \nabla_{\mathcal{H}} df  -\mathfrak{T}_\mathcal{H}^\varepsilon  df \|_{2\varepsilon}^2 +\left( \rho_1-\frac{\kappa}{2\varepsilon}\right)\int_\M \| df \|_{\mathcal{H}^*}^2,
 \]
 and thus, since $Lf=-\lambda_1 f$,
 \begin{align}\label{STT}
 \lambda_1^2 \int_\M f^2 +2\lambda_1 \varepsilon \int_\M \| d f\|_{\mathcal{V}^*}^2   \ge \int_\M \| \nabla_{\mathcal{H}} df  -\mathfrak{T}_\mathcal{H}^\varepsilon  df \|_{2\varepsilon}^2 +\lambda_1 \left( \rho_1-\frac{\kappa}{2\varepsilon}\right)\int_\M f^2.
 \end{align}
 On the other hand, from Lemma \ref{LMM}, we have
 \begin{align*}
\int_\M \| \nabla_{\mathcal{H}} df  -\mathfrak{T}_\mathcal{H}^\varepsilon  df \|_{2\varepsilon}^2 \ge & \int_\M \| \nabla_{\mathcal{H}} df  -\mathfrak{T}_\mathcal{H}^\varepsilon  df \|_{\mathcal{H}^*}^2 +2\varepsilon \int_\M \left\| \nabla_{\mathcal{H}} df  -\frac{3}{2}\mathfrak{T}_\mathcal{H}^\varepsilon  df \right\|_{\mathcal{V}^*}^2 \\
 & +\frac{\rho_2}{2} \int_\M \| d f \|_{\mathcal{V}^*}^2-\frac{5}{8 \varepsilon} \int_\M \| df \|_{\mathcal{H}^*}^2.
\end{align*}
It is readily checked that
\[
 \| \nabla_{\mathcal{H}} df  -\mathfrak{T}_\mathcal{H}^\varepsilon  df \|_{\mathcal{H}^*}^2= \| \nabla^{2,\#}_\mathcal{H} f \|^2+\frac{1}{4} \mathbf{Tr} ( J^*_{\nabla_\mathcal{V}f} J_{\nabla_\mathcal{V}f}),
\]
where $ \nabla^{2,\#}_\mathcal{H} f$ denotes the symmetrization of the horizontal Hessian of $f$. Thus we have
 \begin{align*}
\int_\M \| \nabla_{\mathcal{H}} df  -\mathfrak{T}_\mathcal{H}^\varepsilon  df \|_{2\varepsilon}^2 \ge & \int_\M \| \nabla^{2,\#}_\mathcal{H} f \|^2 +2\varepsilon \int_\M \left\| \nabla_{\mathcal{H}} df  -\frac{3}{2}\mathfrak{T}_\mathcal{H}^\varepsilon  df \right\|_{\mathcal{V}^*}^2  +\frac{3\rho_2}{4} \int_\M \| d f \|_{\mathcal{V}^*}^2-\frac{5}{8 \varepsilon} \lambda_1  \int_\M f^2.
\end{align*}
Choosing $\varepsilon$ such that $2\lambda_1 \varepsilon=\frac{3\rho_2}{4} $ and using the last inequality in \eqref{STT} gives eventually
\[
\left(  \lambda_1^2 -\lambda_1 \left( \rho_1-\frac{9\kappa}{8\varepsilon}\right) \right) \int_\M f^2 \ge  \int_\M \| \nabla^{2,\#}_\mathcal{H} f \|^2+2\varepsilon \int_\M \left\| \nabla_{\mathcal{H}} df  -\frac{3}{2}\mathfrak{T}_\mathcal{H}^\varepsilon  df \right\|_{\mathcal{V}^*}^2.
 \]
 From Cauchy-Schwarz inequality, given the value of $\lambda_1$, we always have
 \[
\left(  \lambda_1^2 -\lambda_1 \left( \rho_1-\frac{9\kappa}{8\varepsilon}\right) \right) \int_\M f^2 \le  \int_\M \| \nabla^{2,\#}_\mathcal{H} f \|^2
\]
This means that, necessarily
\[
\left\| \nabla_{\mathcal{H}} df  -\frac{3}{2}\mathfrak{T}_\mathcal{H}^\varepsilon  df \right\|_{\mathcal{V}^*}=0,
\]
and moreover that $\nabla^{2,\#}_\mathcal{H} f $ is a multiple of $g_\mathcal{H}$. This immediately implies \eqref{equ:Hessian extremal} and \eqref{equ:extremal equations}.
 \end{proof}

We are now in position to prove Theorem \ref{obata}.

\begin{proof}
Let $f \in C^\infty(\M)$ such that $Lf=-\lambda_1 f$ with $ \lambda_1 =  \frac{\rho d}{d-1+3 \di }$. From the previous lemma, we have
\begin{align*}
 \nabla^2 f(X,Y) = -\frac{\lambda_1}{d} f  \langle X,Y \rangle_\mathcal{H} -\frac{1}{2} T(X,Y) f ,\quad \forall X,Y\in\mathcal H .
\end{align*}
and
\begin{align*}
\nabla^2 f (X,Z)=\frac{2 \lambda_1}{d} J_Z(X)f, \quad \forall X \in\mathcal H, Z \in \mathcal{V} .
\end{align*}
The trick is now that, since $\M$ has transverse symmetries, $-L$ commutes with any $Z \in \mathcal{V}$ (see \cite{BG}), and thus $Zf$ is also an eigenfunction for the same eigenvalue $\lambda_1$. In particular $Zf$ also satisfies the equation  \eqref{equ:extremal equations}. This gives for a horizontal vector field $X$ and $Z \in \mathcal{V}$,
\[
\nabla^3f(X,Z,Z)=\frac{4 \lambda^2_1}{d^2}J^2_Z(X)f.
\]
From the $H$-type assumption, we deduce
\[
\nabla^3f(X,Z,Z)=-\frac{4 \lambda^2_1}{d^2}\| Z \|^2_\mathcal{V} Xf.
\]
Taking the trace and using the fact that both $f$ and $Zf$ are eigenfunctions of $-L$ with the same eigenvalue, we deduce that for any $Z \in \mathcal{V}$,
\[
Z^2 f=-\frac{4 \lambda^2_1}{d^2}\| Z \|^2_\mathcal{V} f.
\]
By polarization, it also implies that for every $Z,Z' \in \mathcal{V}$,
\begin{align}\label{hessiangroup}
\frac{1}{2} (ZZ' +Z'Z)f =-\frac{4 \lambda^2_1}{d^2} \langle Z,Z' \rangle_\mathcal{V} f.
\end{align}
Since $\M$ is compact, we easily see that $\mathcal{V}$ is a Lie algebra of compact type. We therefore can choose $g_\mathcal{V}$ to be a bi-invariant metric. We consider then the Riemannian metric on $\M$,
\[
g_{2 \varepsilon}=g_\mathcal{H} \oplus \frac{1}{2 \varepsilon} g_\mathcal{V},
\]
where $\varepsilon=\frac{2 \lambda_1}{d}$. By denoting $\tilde{\nabla}$ the Levi-Civita connection associated to $g_{2\varepsilon}$, it is then an easy exercise to check that the previous relations imply then that for every smooth vector fields $X,Y$
\[
\tilde{\nabla}^2 f(X,Y)=-\frac{\lambda_1}{d} f g_{2 \varepsilon}(X,Y).
\]
As a consequence of Obata's theorem \cite{obata}, we deduce that $(\M, g_{2\varepsilon})$ is isometric to a sphere. Also by the very same Obata's theorem, the relations \eqref{hessiangroup} imply that the Lie group $\mathbb{G}$ generated by $\mathcal{V}$ is a sphere itself. This implies that this group is either $\mathbf{U}(1)$ or $\mathbf{SU}(2)$. Morever, by the very definition of sub-Riemannian manifolds with transverse symmetries, $\mathbb{G}$ is seen to act properly on $\M$. We deduce that there is a Riemannian submersion with totally geodesic fibers
\[
\mathbb{G} \to \M \to \M / \mathbb{G}.
\]
The classification of Riemannian submersions with totally geodesic fibers  of the sphere that was done in Escobales \cite{escobales}  completes our proof.
\end{proof}

\end{document}